\documentclass[12pt,a4paper]{amsart}
\usepackage{amscd,amsfonts,amsmath,amsthm,amssymb,verbatim,mathrsfs}
\usepackage[dvips]{graphicx}
\usepackage{graphics,hyperref}
\usepackage{enumitem}
\usepackage{psfrag}
\usepackage{pstcol,pst-plot,pst-3d}
\usepackage{latexsym}
\usepackage{pstcol,pst-plot,pst-3d}
\usepackage{multicol}

\psset{unit=0.7cm,linewidth=0.8pt,arrowsize=2.5pt 4}

\newpsstyle{fatline}{linewidth=1.5pt}
\newpsstyle{fyp}{fillstyle=solid,fillcolor=verylight}
\definecolor{verylight}{gray}{0.97}
\definecolor{light}{gray}{0.9}
\definecolor{medium}{gray}{0.85}

\def\NZQ{\Bbb}               
\def\NN{{\NZQ N}}

\def\ZZ{{\NZQ Z}}

\def\frk{\frak}               

\def\Phi{{\frk n}}
\def\Phi{{\frk N}}
\def\MS{{\mathcal S}}

\def\ub{{\bold u}}
\def\wb{{\bold w}}
\def\ab{{\bold a}}

\def\cb{{\bold c}}

\def\opn#1#2{\def#1{\operatorname{#2}}} 
\opn\chara{char} \opn\length{\ell} \opn\pd{pd} \opn\rk{rk}
\opn\projdim{proj\,dim} \opn\injdim{inj\,dim} \opn\rank{rank}
\opn\depth{depth} \opn\grade{grade} \opn\height{height}
\opn\embdim{emb\,dim} \opn\codim{codim} \opn\sgn{sgn}

\opn\Tr{Tr} \opn\bigrank{big\,rank}
\opn\superheight{superheight}\opn\lcm{lcm}
\opn\trdeg{tr\,deg}
\opn\reg{reg} \opn\lreg{lreg} \opn\ini{in} \opn\lpd{lpd}
\opn\size{size}\opn\bigsize{bigsize}
\opn\cosize{cosize}\opn\bigcosize{bigcosize}
\opn\sdepth{sdepth}\opn\sreg{sreg}
\opn\link{link}\opn\fdepth{fdepth}
\opn\div{div} \opn\Div{Div} \opn\cl{cl} \opn\Cl{Cl} \opn\Cor{Cor}
\opn\Spec{Spec} \opn\Supp{Supp} \opn\supp{supp} \opn\Sing{Sing}
\opn\Ass{Ass} \opn\Min{Min}\opn\Mon{Mon} \opn\dstab{dstab} \opn\astab{astab}
\opn\Ann{Ann} \opn\Rad{Rad} \opn\Soc{Soc} \opn\Gr{Gr}
\opn\Im{Im} \opn\Ker{Ker} \opn\Coker{Coker} \opn\Am{Am}
\opn\Hom{Hom} \opn\Tor{Tor} \opn\Ext{Ext} \opn\End{End}
\opn\Aut{Aut} \opn\id{id} \opn\span{span}

\opn\nat{nat}
\opn\pff{pf}
\opn\Pf{Pf} \opn\GL{GL} \opn\SL{SL} \opn\mod{mod} \opn\ord{ord}
\opn\Gin{Gin} \opn\Hilb{Hilb}\opn\sort{sort} \opn\Gale{Gale}
\opn\aff{aff} \opn\conv{conv} \opn\relint{relint} \opn\st{st}   \opn\cone{cone}
\opn\lk{lk} \opn\cn{cn} \opn\core{core} \opn\vol{vol}
\opn\link{link} \opn\star{star}\opn\lex{lex} \opn\Gr{Gr}
\opn\gr{gr}

\def\pot#1#2{#1[\kern-0.28ex[#2]\kern-0.28ex]}

\opn\dirlim{\underrightarrow{\lim}}
\opn\inivlim{\underleftarrow{\lim}}
\def\Implies{\ifmmode\Longrightarrow \else
        \unskip${}\Longrightarrow{}$\ignorespaces\fi}
\def\implies{\ifmmode\Rightarrow \else
        \unskip${}\Rightarrow{}$\ignorespaces\fi}
\def\iff{\ifmmode\Longleftrightarrow \else
        \unskip${}\Longleftrightarrow{}$\ignorespaces\fi}

\let\:=\colon
\newtheorem{Theorem}{Theorem}[section]
\newtheorem{Lemma}[Theorem]{Lemma}
\newtheorem{Corollary}[Theorem]{Corollary}
\newtheorem{Proposition}[Theorem]{Proposition}

\newtheorem{Example}[Theorem]{Example}

\newtheorem{Definition}[Theorem]{Definition}

\let\epsilon\varepsilon
\let\kappa=\varkappa
\makeatletter

\def\qed{\ifhmode\textqed\fi
      \ifmmode\ifinner\quad\qedsymbol\else\dispqed\fi\fi}
\def\textqed{\unskip\nobreak\penalty50
       \hskip2em\hbox{}\nobreak\hfil\qedsymbol
       \parfillskip=0pt \finalhyphendemerits=0}
\def\dispqed{\rlap{\qquad\qedsymbol}}

\opn\dis{dis}
\def\pnt{{\raise0.5mm\hbox{\large\bf.}}}

\opn\Lex{Lex}

\begin{document}


\title{On the strongly robustness property of toric ideals}
\author[1]{ Dimitra Kosta }
\author[2] {Apostolos Thoma}
\author[3] {Marius Vladoiu}
\thanks{Corresponding author: Dimitra Kosta}
\address{Dimitra Kosta, School of Mathematics, University of Edinburgh and Maxwell Institute for Mathematical Sciences, United Kingdom }
\email{D.Kosta@ed.ac.uk}

\address{Apostolos Thoma, Department of Mathematics, University of Ioannina, Ioannina 45110, Greece}
\email{athoma@uoi.gr}

\address{Marius Vladoiu, Faculty of Mathematics and Computer Science, University of Bucharest, Str. Academiei 14, Bucharest, RO-010014, Romania, and}
\address{``Simion Stoilow" Institute of Mathematics of the Romanian Academy, P.O. Box 1-764, 014700, Bucharest, Romania}
\email{vladoiu@fmi.unibuc.ro}


\begin{abstract} To every toric ideal one can associate an oriented matroid structure, consisting of a graph and another toric ideal, called bouquet ideal. The connected components of this graph are called bouquets.
Bouquets are of three types; free, mixed and non mixed.
We prove that the cardinality of the following sets - the set of indispensable elements, minimal Markov bases, the Universal Markov basis and
the Universal Gr{\" o}bner basis of a toric ideal -
depends only on the type of the bouquets and the bouquet ideal.  These results enable us to introduce the strongly robustness
simplicial complex and show that it determines the strongly robustness property. For codimension 2 toric ideals, we study the strongly robustness simplicial complex and prove that robustness implies strongly robustness.
\end{abstract}
\maketitle

\section{Introduction}
An ideal in a polynomial ring is called \textit{robust} if it is minimally generated by the Universal Gr{\"o}bner basis  (see \cite{BR}), where the Universal Gr{\"o}bner basis is the union of all reduced  Gr{\"o}bner bases.
For   toric ideals, the Universal  Gr{\"o}bner basis is always contained in the Graver basis, see \cite[Chapter 4]{St}. If $I_A$ is positively graded then every minimal system of generators of a toric ideal $I_A$ is contained in
the Graver basis, see \cite[Theorem 2.3]{CTV2}.
Thus there exists a stronger notion of robustness for positively graded toric ideals. A positively graded toric ideal is called \textit{strongly robust} if it is minimally
generated by its Graver basis, see \cite{S}. It follows  that for a strongly robust toric ideal $I_A$  the
following sets are identical: the set of indispensable elements, any minimal system of generators, any reduced Gr{\"o}bner basis, the Universal Gr{\"o}bner basis and the Graver basis, see for example \cite[Section 4]{PTV}. Robust and strongly robust ideals have been studied in several articles \cite{BZ, B, BR,  BBDLMNS, cng, CNG, G-MT, GP, PTV, St, SZ, S, T}.

 Strongly robust toric ideals appear
in geometry,
as defining ideals of non-pyramidal self-dual projective toric varieties, see \cite{BDR, TV}; in combinatorial commutative algebra,
as toric ideals of Lawrence matrices helping in the  computation of the Graver basis, see \cite[Chapter 7]{St};
in algebraic statistics, as toric ideals whose Markov bases have the distance-reducing property, see \cite{S, TA}.
Note that for a toric ideal being
strongly robust implies being robust. In \cite{BBDLMNS}, Boocher et al. wonder whether robust toric ideals are always strongly robust. In fact, they proved this is the case for
toric ideals of graphs. On the other hand, based on \cite[Theorem 4.2, Remark 4.4]{PTV2}, using the theory of bouquets, \cite{PTV}, it follows that if robust toric ideals of hypergraphs are strongly robust then this is true for all positively graded toric ideals.

Theorem 4.2 of \cite{PTV} suggests a special relation
between  strongly robust ideals and bouquets.
For example, if all the bouquets are mixed then the toric ideal is strongly robust, see \cite[Corollary 4.4]{PTV}, thus Petrovi{\' c} et al. raised the question whether
all strongly robust toric ideals have at least one mixed bouquet. This question was answered in affirmative
by Sullivant \cite{S} for codimension 2 toric ideals based on the
reduced Gale diagrams, a useful tool given  by Peeva and Sturmfels in \cite{PS} for studying problems related to codimension 2 toric ideals. Note that reduced Gale diagrams have a strong relationship with the bouquet structure and they also preserve information about the mixedness or not of the bouquets, see Section 4 of this article.

In Section~\ref{section:bouquets}, we strengthen the bouquet results of \cite{PTV} in the following manner.
To every bouquet we can assign a sign $+$ if it is non-mixed, $-$ if it is mixed and $0$ if it is free. We prove that if two toric ideals  have the same bouquet ideal and the same signatures on the corresponding bouquets, then there is a bijective correspondence between
 indispensables, Markov bases, Gr{\"o}bner bases and the universal Gr{\"o}bner bases   of the two ideals.
 The corresponding result for circuits and the Graver bases was proved in \cite{PTV}.  In Section~\ref{section:stronglyrobust}, the previous results help us to define the strongly robustness simplicial complex, about which we show that determines the strongly robustness property for toric ideals, thus partially answering a question posed by Boocher and Robeva \cite[Question 3.5]{BR}. We show how
one can use the second Lawrence lifting to compute the strongly robustness simplicial complex. Then the theory of generalized Lawrence matrices of \cite[Section 2]{PTV} produces examples of strongly robust toric ideals; actually any strongly robust toric ideal can be produced in this way. In Section~\ref{section:codim2}, we use the results of \cite{PS, S} in codimension 2 to provide a Hilbert basis method to compute the strongly robustness simplicial complex and prove that, for codimension 2 toric ideals, robustness implies strongly robustness. Finally, using \cite[Lemma 2.7]{S}, we prove an interesting property of codimension 2 ideals: if circuits are indispensable then all elements of the Graver basis are indispensable.

 \section{Bouquet Ideals }\label{section:bouquets}
 Let $K$ be a field and $A=[\ab_1,\dots,\ab_n]\in\ZZ^{m\times n}$ be an integer matrix  with the set of column vectors $\{\ab_1,\ldots,\ab_n\}$ such that $\Ker_{\ZZ}(A)\cap \NN^n=\{{\mathbf{0}}\}$. The toric ideal of $A$ is the ideal $I_A\subset K[x_1,\ldots,x_n]$ generated by the binomials ${x}^{{\ub}^+}-{x}^{{\ub}^-}$
 where ${\ub}\in\Ker_{\ZZ}(A),$
and ${\bf u}={\bf u}^+-{\bf u}^-$ is the expression of  ${\bf u}$ as a difference of two non-negative vectors with disjoint support, see \cite[Chapter 4]{St}. The condition $\Ker_{\mathbb{Z}}(A)\cap \mathbb{N}^n=\{{\bf 0}\} $ means that there is a strictly positive vector on the row span of the matrix $A$, therefore the toric ideal is positively graded. To the matrix $A$ we associate its Gale transform, which is the $n\times (n-r)$ matrix whose columns span the lattice $\Ker_{\ZZ}(A)$, and $r$ is the rank of $A$. We will denote the set of ordered row vectors of the Gale transform by $\{{\bf b}_1, \dots, {\bf b}_n\}$. The vector ${\bf a}_i$ is called {\em free} if its Gale transform ${\bf b}_i$ is equal to the zero vector, which means that $i$ is not contained in the support of any element in $\Ker_{\mathbb{Z}}(A)$.
 The {\em bouquet graph} $G_A$ of $I_A$ is the graph on the set of vertices $\{{\bf a}_1,\dots, {\bf a}_n\}$, whose edge set $E_A$ consists of those $\{\ab_i,\ab_j\}$ for which ${\bf b}_i$ is a rational multiple of ${\bf b}_j$ and vice versa. The connected components of the graph $G_A$ are called {\em bouquets}.

It follows from the definition that the free vectors of $A$ form one bouquet, which we call the {\em free bouquet} of $G_A$.  The non-free bouquets are of two types: {\em mixed} and {\em non-mixed}. A non-free bouquet is mixed if contains an edge $\{\ab_i,\ab_j\}$ such that ${\bf b}_i=\lambda {\bf b}_j$ for some $\lambda<0$, and is non-mixed if it is either an isolated vertex or for all of its edges $\{\ab_i,\ab_j\}$ we have ${\bf b}_i=\lambda {\bf b}_j$ with $\lambda>0$, see \cite[Lemma 1.2]{PTV}.

We reorder the vectors $\ab_1,\dots,\ab_n$ to $\ab_{11},\ab_{12}, \dots,\ab_{1k_1},\ \ab_{21},\ab_{22}, \dots,\ab_{2k_2},\ \dots,\\ \ab_{s1},\ab_{s2}, \dots,\ab_{sk_s}$ in such a way the first $k_1$ vectors belong to the bouquet $B_1$, the next $k_2$ to $B_2$ and so on up to the last $k_s$ that belong to the bouquet $B_s$. Note that $k_1+k_2+\dots +k_s=n$.
For each bouquet $B_i$ we define two vectors $\cb_{B_i}$ and $\ab_{B_i}$ which record the bouquet's  types and linear dependencies of the bouquet vectors. If the bouquet $B_i$ is free then we set $\cb_{B_i}\in\ZZ^n$ to be any nonzero vector such that $\supp(\cb_{B_i})=\{i1,\dots , ik_i\}$ and with the property that the first nonzero coordinate, $c_{i1}$, is positive. For a non-free bouquet $B_i$ of $A$, consider  the Gale transforms of the elements in $B_i$. All the Gale transforms are nonzero, since the bouquet is not free, and pairwise linearly dependent, since they belong to the same bouquet. Therefore, there exists a nonzero coordinate $l$ in all of them. Let $g_l=\gcd(({\bf b}_{i1})_l, ({\bf b}_{i2})_l),\dots , ({\bf b}_{ik_i})_l)$, where $({\bf w})_l$ is the $l$-th component of a vector ${\bf w}$. Then ${\bf c}_{B_i}$  is the vector in $\ZZ^n$ whose $qj$-th component  is $0$ if $q\not= i$, and   $c_{ij}=\varepsilon_{i1}({\bf b}_{ij})_l/g_l$, where $\varepsilon_{i1}$ represents the sign of the integer $({\bf b}_{i1})_l$. Note that $c_{i1}$ is always positive. Then the vector $\ab_{B_i}$ (see \cite[Definition 1.7]{PTV}), is defined as $\ab_{B_i}=\sum_{j=1}^n (c_{B_i})_j\ab_j\in\ZZ^m$.

If $B_i$ is a  non-free bouquet of $A$, then $B_i$ is a mixed bouquet if and only if the vector ${\bf c}_{B_i}$ has a negative and a positive coordinate, see \cite[Lemma 1.6]{PTV}, and  $B_i$ is non-mixed if and only if the vector $\cb_{B_i}$ has all nonzero coordinates positive.
Let $A_B$  be the matrix with columns the vectors $\ab_{B_i}$, $1\leq i\leq s$ then the toric ideal $I_{A_B}$ is the bouquet ideal  of $A$.

Let ${\mathbf u}=(u_1, u_2, \ldots, u_s)\in  \Ker_{\ZZ}({A_B})$
then the linear map $$D({\mathbf u})=(c_{11}u_1, c_{12}u_1,\ldots ,c_{1k_1}u_1, c_{21}u_2, \ldots ,c_{2k_2}u_2, \ldots ,c_{s1}u_s, \ldots ,c_{sk_s}u_s),   $$
where all $c_{j1}, 1 \leq j \leq s$, are positive, is an isomorphism from $ \Ker_{\ZZ}({A_B})$ to $ \Ker_{\ZZ}(A)$, see \cite[Theorem 1.9]{PTV}.
Therefore, since the matrices $A$ and $A_B$ have their kernels isomorphic, their toric ideals $I_A$ and $I_{A_B}$ have the same codimension.

 Note that the bouquet ideal is simple: a toric ideal is called \textit{simple} if every bouquet is a singleton, in other words if $I_T\subset K[x_1,\dots , x_s]$ and has $s$ bouquets. The bouquet ideal of a simple toric ideal $I_A$ is $I_A$. All toric
ideals $I_A\subset K[x_1,\dots , x_n]$ of height one, the ideals of monomial curves,  are simple if $n\geq 3$. There are no simple toric ideals in codimension one, except the zero ideal in the polynomial ring in one variable.

 \begin{Definition}
 Let $I_T\subset K[x_1,\dots , x_s]$ be a simple toric ideal and $\omega \subset \{1, \dots , s\}$. A toric ideal $I_A$
 is called $T_{\omega}$-robust ideal if and only if
 \begin{itemize}
  \item the bouquet ideal of $I_A$ is $I_T$ and
  \item $\omega =\{i\in [s]| B_i \ \text{is non mixed}\}.$
 \end{itemize}
 \end{Definition}

The $T_{\emptyset }$-robust toric ideals are always strongly robust ideals, see \cite[Section 4]{PTV}, while the $T_{[s]}$-robust toric ideals are the
stable ideals, see \cite[Section 3]{PTV}. Every toric ideal $I_A$ is a $T_{\omega}$-robust ideal where $I_T$ is the bouquet ideal of $I_A$
and $\omega$ is the set of indices corresponding to the non mixed bouquets of $I_A$. The complement of $\omega$ in $[s]$ is the set of indices corresponding to the free and the mixed bouquets. The ideal $I_A$ has a free bouquet if
and only if the bouquet ideal $I_T$ has a free
bouquet, consisting of one free vector since $I_T$ is simple. Thus the combination of the
information given by $\omega$ and $T$ provides
the signatures of all the bouquets of a $T_{\omega}$-robust toric ideal.
The results of this section can be summarized:  the cardinality of the sets of different toric bases depend
only on the signatures of the bouquets and the bouquet ideal. Remark that in \cite[Theorem 1.11]{PTV} it was proved that the cardinality of the set of circuits and of the Graver basis, respectively, depends only on the bouquet ideal. Also in the case of any two $T_{[s] }$-robust toric ideals there is a one-to-one  correspondence between all different toric bases, \cite[Theorem 3.7]{PTV}.
The aim of this section is to generalize these
results for all toric ideals with the same signatures and the same bouquet ideal, by proving that in these cases there exists also a one-to-one correspondence between the indispensable elements, minimal Markov bases, the universal Markov bases (see Theorem \ref{markov}), reduced Gr{\"o}bner bases and the universal   Gr{\"o}bner bases (see Theorem \ref{Universal}).

A \textit{Markov basis} of $A$ is a finite subset $\mathcal{M}$ of $\Ker_{\mathbb{Z}}(A)$ such that whenever $\textbf{w}, \textbf{v} \in \mathbb{N}^n$
and $\textbf{w}- \textbf{v} \in \Ker_{\mathbb{Z}}(A)$ there exists a subset $\{\textbf{u}_i : i = 1, \ldots , r\}$ of $\mathcal{M}$ that \textit{connects} $\textbf{w}$
to $\textbf{v}$. Here, connectedness means that $\textbf{w}- \textbf{v} = \sum_{i=1}^{r} \textbf{u}_i$, and $(\textbf{w}-\sum_{i=1}^{p} \textbf{u}_i) \in \mathbb{N}^n$ for all $1 \leq p \leq r$ and $\textbf{u}_i \in \mathcal{M}$ means that  $\textbf{u}_i \in \mathcal{M}$ or the opposite  $-\textbf{u}_i \in \mathcal{M}$. (We will also express that $\textbf{w}\in\mathbb{N}^n$ by writing $\textbf{w}\geq \textbf{0}$, that is, each component $w_i$ of $\wb$ is non-negative.)  A Markov basis $\mathcal{M}$ is minimal if no proper subset of $\mathcal{M}$ is a Markov basis.
The universal Markov basis $\mathcal{M}(A)$ of $A$ is the union of all minimal Markov bases of $A$, up to identification of opposite vectors, and the set of indispensable elements $S(A)$ is the intersection of all different minimal Markov bases, via the same identification. By a classical result of Diaconis and Sturmfels
\cite{DS} minimal Markov bases of $A$ correspond to minimal system of generators for the toric ideal $I_A$: $\mathcal{M}$ is minimal Markov basis  of $A$ if and only if $\{\textbf{x}^{u^+}-\textbf{x}^{u^-}|\textbf{u}\in \mathcal{M} \} $ is a minimal system of binomial generators for the toric ideal $I_A$. The set of indispensable binomials  $S(I_A)$ of the toric ideal is the set  $\{\textbf{x}^{u^+}-\textbf{x}^{u^-}|\textbf{u}\in S(A) \}$, that is the intersection  of all minimal systems of binomial generators of $I_A$, up to identification of opposite binomials.

For the proofs of the next two theorems we need the following notation. Let $ B_{1}^{(j)}, \ldots, B_{s}^{(j)}$ be the bouquets of the toric ideal $I_{A_j}$  with $\cb_{B_1^{(j)}},\ldots, \cb_{B_s^{(j)}}$ the corresponding vectors for $j=1,2$. Furthermore, we denote by $k_i$ and $l_i$, respectively, the number of elements of the $i$-th bouquet of $B_i^{(1)}$ and $B_i^{(2)}$, respectively, and by $c^{(1)}_{i1},\ldots,c^{(1)}_{ik_i}$ and $c^{(2)}_{i1},\ldots,c^{(2)}_{il_i}$ the nonzero coordinates of $\cb_{B_i^{(1)}}$ and $\cb_{B_i^{(2)}}$, respectively. Furthermore, associated to the matrices $A_1$ and $A_2$ we consider the toric ideals $I_{A_1}\subset K[x_{11},\ldots,x_{1k_1},\ldots, x_{s1},\ldots,x_{sk_s}]$ and $I_{A_2}\subset K[y_{11},\ldots,y_{1l_1},\ldots, y_{s1},\ldots,y_{sl_s}]$.  According to \cite[Theorem 1.9]{PTV}, for each $j=1,2$ there is a one-to-one correspondence between the elements of $\Ker_{\mathbb{Z}} (T)$ and the elements of $\Ker_{\mathbb{Z}} (A_j)$  given by the map $D_j: \textbf{u} \mapsto D_j(\textbf{u})$, which sends the element $\textbf{u} = (u_1, \ldots, u_s) \in  \Ker_\mathbb{Z}(T)$  to
$D_j(\textbf{u}):= \cb_{B_{1}^{(j)}}u_1 + \cdots + \cb_{B_{s}^{(j)}}u_s\in\Ker_{\ZZ}(A_j)$. Therefore, all elements of $\Ker_{\mathbb{Z}} (A_j)$ are in the form $D_j(\textbf{u})$ for a unique $\textbf{u}\in \Ker_{\mathbb{Z}}(T)$.

\begin{Theorem}\label{markov}
  Let $I_{A_1}, I_{A_2}$ be two $T_{\omega}$-robust toric ideals. If $\mathcal{M}_1$ is a minimal Markov basis of $A_1$ then
$\mathcal{M}_2=\{D_2( \textbf{u}) | \  \ub\in\Ker_{\ZZ}(T) \text{ such that } D_1( \textbf{u}) \in \mathcal{M}_1\}$ is a minimal Markov basis of $A_2$.
  Thus there exists a one-to-one correspondence between
  the sets of all different minimal Markov bases of the two ideals. In particular, there exists a one-to-one correspondence between the indispensable elements  and the universal Markov bases of ${A_1}$ and ${A_2}$, respectively.
\end{Theorem}

\begin{proof}

Let $\mathcal{M}_1$ be a minimal Markov basis of $A_1$.   First we want to prove that
$\mathcal{M}_2=\{D_2( \textbf{u}) | \  \ub\in\Ker_{\ZZ}(T) \text{ such that } D_1( \textbf{u}) \in \mathcal{M}_1\}$ is a Markov basis of $A_2$. It is enough to prove that one can connect $D_2(\textbf{v})^+$ to $D_2(\textbf{v})^-$ for every $\textbf{v} \in \Ker_\mathbb{Z}(T)$.

Since $\mathcal{M}_1$ is a Markov basis of  $A_1$
 there exists a subset $\{ D_1(\textbf{u}_j) | \ j=1, \ldots r \}$ of $\mathcal{M}_1$ such that for the element $D_1(\textbf{v} ) \in \Ker_\mathbb{Z}(A_1)$ we have
$D_1(\textbf{v} )^{+} - D_1(\textbf{v})^{-} = \sum_{j=1}^{r} D_1(\bf{u}_j)$ and
$D_1(\textbf{v})^{+} - \sum_{j=1}^{p} D_1(\textbf{u}_j) \geq \textbf{0}$ for all $1 \leq p \leq r$.
Therefore,  we have that
$$ \left(\cb_{B_{i}^{(1)}}v_i \right)^+ - \left( \cb_{B_{i}^{(1)}} u_{1i} + \cdots + \cb_{B_{i}^{(1)}} u_{pi} \right)\geq \textbf{0} \ \ \ \ \forall \ 1\leq p\leq r, \ \forall \ 1\leq i\leq s, $$
which we can further write at coordinate level as
$$ \left( \left(c_{i1}^{(1)}, c_{i2}^{(1)}, \ldots, c_{ik_i}^{(1)} \right) v_i\right )^+ \geq \left(c_{i1}^{(1)}, c_{i2}^{(1)}, \ldots, c_{ik_i}^{(1)} \right) u_{1i} + \cdots + \left(c_{i1}^{(1)}, c_{i2}^{(1)}, \ldots, c_{ik_i}^{(1)} \right) u_{pi},$$
for all $1\leq p\leq r$ and for all $1\leq i\leq s$.

\underline{Case 1:} Let $i \in \omega, v_i >0$. In this case the bouquet $B_{i}^{(1)}$ is non-mixed and all $c_{ij}^{(1)}>0$. We have for all $1\leq p\leq r$ and for all $1\leq j\leq k_i$
$$c^{(1)}_{ij} v_i  \geq  c^{(1)}_{ij}  \left( u_{1i} + \cdots + u_{pi} \right)$$
which gives
$$
v_i  \geq   u_{1i} + \cdots + u_{pi} .$$
Since the toric ideal  $I_{A_2}$ is $T_{\omega}$-robust, $B_{i}^{(2)}$ is also non-mixed and all $c_{ij}^{(2)}>0$. Therefore, the last set of inequalities becomes by multiplication with $c_{ij}^{(2)}$
$$c^{(2)}_{ij} v_i  \geq  c^{(2)}_{ij}  ( u_{1i} + \cdots + u_{pi} )\ \ \ \ \ \forall \ 1\leq j\leq l_i ,$$
which implies
$$ \left( \left(c_{i1}^{(2)}, c_{i2}^{(2)}, \ldots, c_{il_i}^{(2)} \right) v_i \right)^+ \geq \left(c_{i1}^{(2)}, c_{i2}^{(2)}, \ldots, c_{il_i}^{(2)} \right) u_{1i} + \cdots + \left(c_{i1}^{(2)}, c_{i2}^{(2)}, \ldots, c_{il_i}^{(2)} \right) u_{pi},$$
for all $1\leq p\leq r$.

\underline{Case 2:} Let $i\in \omega, v_i \leq 0$. Here again the bouquet $B_{i}^{(1)}$ is non-mixed and $c_{ij}^{(1)}>0$ for all $j=1,2, \ldots, k_i$. We have
$$ 0  \geq  c^{(1)}_{ij}  ( u_{1i} + \cdots + u_{pi} )$$
which gives
$$ 0  \geq   u_{1i} + \cdots + u_{pi} .$$
Again the toric ideal  $I_{A_2}$ is $T_{\omega}$-robust, $B_{i}^{(2)}$ is also non-mixed and all $c_{ij}^{(2)}>0$, $j=1,2, \ldots, l_i$. Therefore, the last set of inequalities becomes by multiplication with $c_{ij}^{(2)}$
$$ 0  \geq  c_{ij}^{(2)} ( u_{1i} + \cdots + u_{pi} ) \ \ \ \ \forall \ 1\leq j\leq l_i,$$
which implies
$$\left( \left(c_{i1}^{(2)}, c_{i2}^{(2)}, \ldots, c_{il_i}^{(2)} \right) v_i\right)^+ \geq \left(c_{i1}^{(2)}, c_{i2}^{(2)}, \ldots, c_{il_i}^{(2)} \right) u_{1i} + \cdots + \left(c_{i1}^{(2)}, c_{i2}^{(2)}, \ldots, c_{il_i}^{(2)} \right) u_{pi},$$
for all $1\leq p\leq r$.

\underline{Case 3:} Let $i\notin \omega$ and the bouquet $B_{i}^{(1)}$ is  the free bouquet of $I_{A_1}$, if it has a free bouquet. Since $\textbf{v}  \in  \Ker_\mathbb{Z}(T)$ and  $B_{i}^{(1)}$
is the free bouquet then $v_i$ is zero. Since this is true for all elements of  $\Ker_\mathbb{Z}(T)$
it means $I_{A_2}$ has a free bouquet, the $B_{i}^{(2)}$, and all
$v_i=u_{1i}=\dots=u_{pi}=0.$
Thus
$$ \left( \left(c_{i1}^{(2)}, c_{i2}^{(2)}, \ldots, c_{il_i}^{(2)} \right) v_i \right)^+ = \left(c_{i1}^{(2)}, c_{i2}^{(2)}, \ldots, c_{il_i}^{(2)} \right) u_{1i} + \cdots + \left(c_{i1}^{(2)}, c_{i2}^{(2)}, \ldots, c_{il_i}^{(2)} \right) u_{pi}={\bf 0}.$$

\underline{Case 4:} Let $i\notin \omega, v_i >0 $ and the bouquet $B_{i}^{(1)}$ is mixed, namely the vector $\cb_{B_{i}^{(1)}}$ has both positive and negative coordinates and splits as $\cb_{B_{i}^{(1)}}= \left(\cb_{B_{i}^{(1)}} \right)^+ - \left(\cb_{B_{i}^{(1)}} \right)^-$. We have
$$\left(\cb_{B_{i}^{(1)}} \right)^+ v_i  \geq  \left(\cb_{B_{i}^{(1)}} \right)^+  \left( u_{1i} + \cdots + u_{pi} \right)$$
and
$$ {\bf 0}  \geq  - (\cb_{B_{i}^{(1)}})^- (u_{1i} + \cdots + u_{pi}).$$
Combined together they give
$ v_i  \geq   u_{1i} + \cdots + u_{pi} \geq 0.$
Since the toric ideals $I_{A_1}$ and $I_{A_2}$ are $T_{\omega}$-robust, whenever $B_{i}^{(1)}$ is mixed, $B_{i}^{(2)}$ is also mixed. In this case $\cb_{B_{i}^{(2)}}$ has both positive and negative coordinates and splits as $\cb_{B_{i}^{(2)}}= \left(\cb_{B_{i}^{(2)}} \right)^+ - \left(\cb_{B_{i}^{(2)}} \right)^-$.
We have
$$\left(\cb_{B_{i}^{(2)}} \right)^+ v_i  \geq  \left(\cb_{B_{i}^{(2)}} \right)^+  ( u_{1i} + \cdots + u_{pi} )$$
and
$$ {\bf 0} \geq  - \left(\cb_{B_{i}^{(2)}} \right)^- (u_{1i} + \cdots + u_{pi})$$ which imply
$$ \left( \left(c_{i1}^{(2)}, c_{i2}^{(2)}, \ldots, c_{il_i}^{(2)} \right) v_i \right)^+ \geq \left(c_{i1}^{(2)}, c_{i2}^{(2)}, \ldots, c_{il_i}^{(2)} \right) u_{1i} + \cdots + \left(c_{i1}^{(2)}, c_{i2}^{(2)}, \ldots, c_{il_i}^{(2)} \right) u_{pi},$$
for all $1\leq p\leq r$.

\underline{Case 5:} Let $i\notin \omega, v_i \leq 0 $ and the bouquet $B_{i}^{(1)}$ is  mixed, namely the vector $\cb_{B_{i}^{(1)}}$ has both positive and negative coordinates and splits as $\cb_{B_{i}^{(1)}}= \left(\cb_{B_{i}^{(1)}} \right)^+ - \left(\cb_{B_{i}^{(1)}} \right)^-$.
We have
$$-\left(\cb_{B_{i}^{(1)}} \right)^- v_i  \geq  -\left(\cb_{B_{i}^{(1)}}\right)^-  ( u_{1i} + \cdots + u_{pi} )$$
and
$$ \textbf{0}  \geq   \left(\cb_{B_{i}^{(1)}}\right)^+ (u_{1i} + \cdots + u_{pi}).$$
Combined together they give
$ v_i  \leq   u_{1i} + \cdots + u_{pi} \leq 0.$
By multiplying appropriately the above inequalities by either $\left(\cb_{B_{i}^{(2)}} \right)^+$ or $\left(\cb_{B_{i}^{(2)}} \right)^-$ we have
$$-\left(\cb_{B_{i}^{(2)}} \right)^- v_i  \geq  -\left(\cb_{B_{i}^{(2)}} \right)^-  ( u_{1i} + \cdots + u_{pi} )$$
and
$$ \textbf{0}  \geq   \left(\cb_{B_{i}^{(2)}} \right)^+ (u_{1i} + \cdots + u_{pi})$$
which imply
$$ \left( \left(c_{i1}^{(2)}, c_{i2}^{(2)}, \ldots, c_{il_i}^{(2)} \right) v_i \right)^+ \geq \left(c_{i1}^{(2)}, c_{i2}^{(2)}, \ldots, c_{il_i}^{(2)} \right) u_{1i} + \cdots + \left(c_{i1}^{(2)}, c_{i2}^{(2)}, \ldots, c_{il_i}^{(2)} \right) u_{pi},$$
for all $1\leq p\leq r$.

Finally combining all five cases
we arrive at
$$\left( D_2(\textbf{v})^{+} - \sum_{j=1}^{p} D_2(\textbf{u}_j) \right) \geq 0 \text{ for all }1 \leq p \leq r .$$ Also, we have that
$D_1(\textbf{v})^{+} - D_1(\textbf{v})^{-} = D_1(\textbf{v}) = \sum_{j=1}^{r} D_1(\textbf{u}_j) \Rightarrow D_2(\textbf{v} )^{+} - D_2(\textbf{v})^{-} = D_2(\textbf{v}) = \sum_{j=1}^{r} D_2(\bf{u}_j),$ since both $D_1, D_2$ are linear maps.

Therefore, this concludes that for any element $D_2(\textbf{v} ) \in \Ker_\mathbb{Z}(A_2)$, there exists a subset $\{ D_2(\textbf{u}_j) | \ j=1, \ldots r \}$ of $\mathcal{M}_2$ that connects $D_2(\textbf{v} )^{+}$ and $D_2(\textbf{v})^{-}$.

This one-to-one correspondence between Markov bases of $A_1$ and $A_2$, ensures also that the map $f: D_1(\textbf{u}) \mapsto D_2(\bf{u})$ preserves the minimality of the Markov bases. Indeed, if $\mathcal{M}_1$ is a minimal Markov basis of $A_1$
then, as we saw, $\mathcal{M}_2 =\{ D_2(\textbf{u}) | \ \ub\in\Ker_{\ZZ}(T) \text{ such that } D_1(\bf{u}) \in \mathcal{M}_1 \}$  is a Markov basis of $A_2$, and if it were not minimal then a proper subset $\mathcal{M}_2^{'}$ of $\mathcal{M}_2$ would be a minimal Markov basis, and thus by repeating the previous argument we would obtain that a
proper subset $\mathcal{M}_1^{'}=\{D_1({\bf u})| \ \ub\in\Ker_{\ZZ}(T) \text{ such that } D_2({\bf u})\in \mathcal{M}_2^{'}\}$ of the minimal Markov basis $\mathcal{M}_1$ would be a Markov basis of $A_1$, a contradiction.

The set of indispensable elements is the intersection of all different minimal Markov bases and the universal Markov basis is the union of all minimal Markov bases. Therefore it  follows that  $f: D_1(\textbf{u}) \mapsto D_2(\bf{u})$ is a one-to-one correspondence between the two sets of the indispensable elements and the same is true for the universal Markov bases of $A_1$ and $A_2$. \qed

\end{proof}

\begin{Theorem}\label{Universal}
  Let $I_{A_1}, I_{A_2}$ be two $T_{\omega}$-robust toric ideals.
  If $\mathcal{G}$ is a reduced Gr{\"o}bner basis  of $I_{A_1}$ with respect to  a monomial order $<_1$ then there exists a monomial order $<_2$ such that
$\mathcal{G}'=\{y^{D_2( \textbf{u})^+}-y^{D_2( \textbf{u})^-} | \ \ub\in\Ker_{\ZZ}(T) \text{ s.t. } \ x^{D_1( \textbf{u})^+}-x^{D_1( \textbf{u})^-} \in \mathcal{G}\}$ is a reduced Gr{\"o}bner basis  of $I_{A_2}$.  Furthermore, if $$\ini_{<_1}(I_{A_1})=(x^{D_1({\bf u_1})^+}, \ldots, x^{D_1({\bf u_r})^+})$$
then $$\ini_{<_2}(I_{A_2})=(y^{D_2({\bf u_1})^+}, \ldots, y^{D_2({\bf u_r})^+}),$$  for some ${\bf u}_1, \ldots ,{\bf u}_r\in \Ker_{\ZZ}(T)$.
Thus there exists   a one-to-one correspondence between the  reduced Gr{\"o}bner bases of the two ideals. In particular, there is a one-to-one correspondence between the universal Gr{\"o}bner bases of $I_{A_1}$ and $I_{A_2}$.
\end{Theorem}

\begin{proof}

It is enough to prove that if  $\mathcal G=\{x^{D_1({\bf u}_1)^+}-x^{D_1({\bf u}_1)^-}, \ldots , x^{D_1({\bf u}_r)^+}-x^{D_1({\bf u}_r)^-} \}$ is a reduced Gr{\"o}bner basis of the toric ideal $I_{A_1}$ with respect to the monomial order $<_1$ on $K[x_{11},\ldots,x_{1{k_1}}, \ldots, x_{s1},\ldots,x_{s{k_s}}]$ for some ${\bf u}_1, \ldots ,{\bf u}_r\in \Ker_{\mathbb{Z}}(T)$,
 then there exists a monomial order $<_2$ on $K[y_{11},\ldots,y_{1l_1},\ldots, y_{s1},\ldots, y_{sl_s}]$ with respect to which  $$\mathcal G'=\left \{y^{D_2({\bf u}_1)^+}-y^{D_2({\bf u}_1)^-}, \ldots ,y^{D_2({\bf u}_r)^+}-y^{D_2({\bf u}_r)^-} \right \}$$ is a reduced Gr{\"o}bner basis of the toric ideal $I_{A_2}$.

We recall that the nonzero coordinates of $\cb_{B_i^{(1)}}$ are  $c_{i1}^{(1)}, c_{i2}^{(1)}, \ldots, c_{ik_i}^{(1)}$. We write each vector $\cb_{B_i^{(1)}}$ as $\cb_{B_i^{(1)}}^+-\cb_{B_i^{(1)}}^-$, where for $i\in \omega$ we have $\cb_{B_i^{(1)}}^-=\mathbf{0}$.
The coordinates of $\cb_{B_i^{(1)}}^+$ and $\cb_{B_i^{(1)}}^-$ are denoted by $c_{ij}^{(1)+}$ and $c_{ij}^{(1)-}$, respectively. One of them is zero and the other $|c_{ij}^{(1)}|$. Similarly, one defines the corresponding notation  for the coordinates of $\cb_{B_i^{(2)}}$.

By \cite[Proposition 1.11]{St}, there exists a weight vector $\wb\in\NN^{k_1+\cdots+k_s}$, say  $\wb=(w_{11}, \ldots, w_{1k_1}, \ldots, w_{s1}, \ldots, w_{sk_s})$ such that $\ini_{<_1}(I_{A_1})=\ini_{\mathbf{w}}(I_{A_1})$.
 Without loss of generality we may assume that $x^{D_1({\bf u}_p)^+}>_1x^{D_1({\bf u}_p)^-}$, i.e.
 the dot product $\mathbf{w} \cdot D_1({\bf u}_p)>0$, for any $p=1,\ldots,r$. Thus
 $$\dots+ \left(c_{i1}^{(1)+}w_{i1}+\dots +c_{i{k_i}}^{(1)+}w_{i{k_i}} \right)u_{pi}-\left (c_{i1}^{(1)-}w_{i1}+\dots +c_{i{k_i}}^{(1)-}w_{i{k_i}} \right)u_{pi}+\dots >0.$$
  We denote by $\cb_{B_i^{(1)}}^+\cdot \mathbf{w}=c_{i1}^{(1)+}w_{i1}+c_{i2}^{(1)+}w_{i2}+\dots +c_{ik_i}^{(1)+}w_{ik_i}$ and by
   $\cb_{B_i^{(1)}}^-\cdot \mathbf{w}=c_{i1}^{(1)-}w_{i1}+c_{i2}^{(1)-}w_{i2}+\dots +c_{ik_i}^{(1)-}w_{ik_i}$. Next, if we let
  $r^{+}_{i}$ be the number of positive elements in $\cb_{B_i^{(2)}}$ and $r^{-}_i$ the number of negative elements in $\cb_{B_i^{(2)}}$, then we define the vector $\mathbf{w}'\in\mathbb{N}^{l_1+\cdots+l_s}$ such that its $it$-th coordinate is equal to $\frac{\cb_{B_i^{(1)}}^+\cdot \mathbf{w}}{c^{(2)}_{it}r^{+}_{i}} $ if $c^{(2)}_{it}>0$ and $\frac{\cb_{B_i^{(1)}}^-\cdot \mathbf{w}}{-c^{(2)}_{it}r^{-}_{i}} $ if
$c^{(2)}_{it}<0$. Note that always ${r_i}^+\geq 1$, but
${r_i}^-=0$ for non mixed bouquets. In this last case there is no reason to define the number $\frac{\cb_{B_i^{(1)}}^-\cdot \mathbf{w}}{-c^{(2)}_{it}r^{-}_{i}} $ since
$c^{(2)}_{it}$ is always greater than $0$.

Let $\prec$ be an arbitrary monomial order on $K[y_{11},\ldots,y_{1l_1},\ldots,y_{s1},\ldots,y_{sl_s}]$. We define by $<_2$ the monomial
order $\prec_{\mathbf{w}'}$ on $K[y_{11},\ldots,y_{1l_1},\ldots,y_{s1},\ldots,y_{sl_s}]$ induced by $\prec$ and $\mathbf{w}'$, see \cite[Chapter 1]{St} for definition.
Next we prove that $x^{D_1({\bf u})^+}>_1x^{D_1({\bf u})^-}$ implies that $y^{D_2({\bf u})^+}>_2y^{D_2({\bf u})^-}$.
Indeed, since $x^{D_1({\bf u})^+}>_1x^{D_1({\bf u})^-}$ then $\mathbf{w}\cdot D_1({\bf u})>0$, where ${\bf u}=(u_1,\ldots,u_s)$.  It follows now from the definition of $\mathbf{w}'$ that
$$ \mathbf{w}'\cdot D_2({\bf u})=\mathbf{w}\cdot D_1({\bf u})>0, $$
which implies $y^{D_2({\bf u})^+}>_2y^{D_2({\bf u})^-}$, as desired.
In particular, we obtain that $y^{D_2({\bf u_p})^+}>_2y^{D_2({\bf u_p})^-}$ for all $p=1, \ldots, r$, and as a consequence
$$\left (y^{D_2({\bf u_1})^+}, \ldots, y^{D_2({\bf u_r})^+}\right )\subset\ini_{<_2}(I_{A_2}).$$
For the converse inclusion, let $y^{D_2({\bf u})^+}- y^{D_2({\bf u})^-}\in I_{A_2}$ be an arbitrary element,
and say that $y^{D_2({\bf u})^+}>_2 y^{D_2({\bf u})^-}$, the other case being similar.
It follows from the previous considerations that $x^{D_1({\bf u})^+}>_1x^{D_1({\bf u})^-}$
and thus there exists an integer $p$ such that $x^{D_1({\bf u_p})^+}|x^{D_1({\bf u})^+}$. That means $x_{ij}^{(D_1({\bf u_p})^+)_{ij}}\big|x_{ij}^{(D_1({\bf u})^+)_{ij}}$ for every $i$, $j$, where $(D_1({\bf u_p})^+)_{ij}$ and $(D_1({\bf u})^+)_{ij}$ represent the $ij$-th coordinate of the vectors $D_1({\bf u_p})^+$ and  $D_1({\bf u})^+$, respectively. The exponent of $x_{ij}$ in $x^{D_1({\bf u_p})^+}$ is $c_{ij}^{(1)+}u_{pi}^++c_{ij}^{(1)-}u_{pi}^-$, while in $x^{D_1({\bf u})^+}$ is $c_{ij}^{(1)+}u_{i}^++c_{ij}^{(1)-}u_{i}^-$. We consider the following cases:\\
\underline{Case 1:} Let $i \in \omega, u_i >0$. In this case the bouquets
$B_{i}^{(1)}$, $B_{i}^{(2)}$  are non-mixed and all $c_{ij}^{(1)}>0$,
$c_{it}^{(2)}>0$. Then
$x_{ij}^{c_{ij}^{(1)+}u_{pi}^++c_{ij}^{(1)-}u_{pi}^-}\Big|x_{ij}^{c_{ij}^{(1)+}u_{i}^++c_{ij}^-u_{i}^-}$
means $x_{ij}^{c_{ij}^{(1)+}u_{pi}^++0u_{pi}^-}\Big|x_{ij}^{c_{ij}^{(1)+}u_{i}^++0u_i^-}$.
Thus $u_{pi}^+\leq u_i$. Therefore
$y_{it}^{c_{it}^{(2)+}u_{pi}^++0u_{pi}^-}\Big|y_{it}^{c_{it}^{(2)+}u_{i}^++0}$ and thus $y_{it}^{(D_2({\bf u_p})^+)_{it}}\big|y_{it}^{(D_2({\bf u})^+)_{it}}$. \\
\underline{Case 2:} Let $i \in \omega, u_i \leq 0$. In this case the bouquets $B_{i}^{(1)}$, $B_{i}^{(2)}$  are non-mixed and all $c_{ij}^{(1)}>0$, $c_{it}^{(2)}>0$. Then $x_{ij}^{c_{ij}^{(1)+}u_{pi}^++c_{ij}^{(1)-}u_{pi}^-}\Big|x_{ij}^{c_{ij}^{(1)+}u_{i}^++c_{ij}^{(1)-}u_{i}^-}$
implies that $x_{ij}^{c_{ij}^{(1)+}u_{pi}^++0u_{pi}^-}$ divides $x_{ij}^{c_{ij}^{(1)+}0+0u_{i}^-}=1$, and thus $u_{pi}^+=0$. Therefore $1=y_{it}^{c_{it}^{(2)+}0+0u_{pi}^-}\Big| y_{it}^{c_{it}^{(2)+}u_{i}^++0u_i^-}=1$ and consequently $y_{it}^{(D_2({\bf u_p})^+)_{it}}\big|y_{it}^{(D_2({\bf u})^+)_{it}}$.\\
\underline{Case 3:} Let $i \not \in \omega$,
and $B^{(1)}_i,B^{(2)}_i$ are free.
In this case, since $\bf{u_p},\bf{ u}\in\Ker_{\ZZ}(T)$, then both $u_{pi}, u_i$ are zero and consequently  $1=y_{it}^{c_{it}^{(2)+}u_{pi}^++c_{it}^{(2)-}u_{pi}^-}\Big|y_{it}^{c_{it}^{(2)+}u_{i}^++c_{it}^{(2)-}u_{i}^-}=1$. Therefore, we obtain $y_{it}^{(D_2({\bf u_p})^+)_{it}}\big|y_{it}^{(D_2({\bf u})^+)_{it}}$.\\
\underline{Case 4:} Let $i \not \in \omega$, $u_i >0$, and $B^{(1)}_i,B^{(2)}_i$ are mixed.
In this case there is at least one $c_{ij}^{(1)}>0$ and at least one $c_{ik}^{(1)}<0$.
For $ij$ such that $c_{ij}^{(1)}>0$  the relation
$x_{ij}^{c_{ij}^{(1)+}u_{pi}^++c_{ij}^{(1)-}u_{pi}^-}\Big|x_{ij}^{c_{ij}^{(1)+}u_{i}^++c_{ij}^{(1)-}u_{i}^-}$
implies $x_{ij}^{c_{ij}^{(1)+}u_{pi}^++0u_{pi}^-}\Big|x_{ij}^{c_{ij}^{(1)+}u_{i}^++0u_i^-}$.
Thus we have $u_{pi}^+\leq u_i$. For $ik$ such that $c_{ik}^{(1)}<0$  the relation
$x_{ik}^{c_{ik}^{(1)+}u_{pi}^++c_{ik}^{(1)-}u_{pi}^-}\Big|x_{ik}^{c_{ik}^{(1)+}u_{i}^++c_{ik}^{(1)-}u_{i}^-}$
implies $x_{ik}^{0u_{pi}^++c_{ik}^{(1)-}u_{pi}^-}\Big|x_{ik}^{0u_{i}^++c_{ik}^{(1)-}0}=1$, and then $u_{pi}^-=0$. Summing up, we have obtained $0\leq u_{pi}^+=u_{pi}\leq u_i$. Therefore
$y_{it}^{c_{it}^{(2)+}u_{pi}^++c_{it}^{(2)-}u_{pi}^-}\Big|y_{it}^{c_{it}^{(2)+}u_{i}^++c_{it}^{(2)-}u_{i}^-}$, and thus $y_{it}^{(D_2({\bf u_p})^+)_{it}}\big|y_{it}^{(D_2({\bf u})^+)_{it}}$. \\
\underline{Case 5:} Let $i\not\in\omega$, $u_i\leq 0$, and $B^{(1)}_i,B^{(2)}_i$ are mixed. In this case there is at least one $c_{ij}^{(1)}>0$ and at least one $c_{ik}^{(1)}<0$.
For $ij$ such that $c_{ij}^{(1)}>0$  the relation
$x_{ij}^{c_{ij}^{(1)+}u_{pi}^++c_{ij}^{(1)-}u_{pi}^-}\Big|x_{ij}^{c_{ij}^{(1)+}u_{i}^++c_{ij}^{(1)-}u_{i}^-}$
implies $x_{ij}^{c_{ij}^{(1)+}u_{pi}^++0u_{pi}^-}\Big|x_{ij}^{c_{ij}^{(1)+}0+0u_{i}^-}=1$.
Thus we have $u_{pi}^+=0$.
For $ik$ such that $c_{ik}^{(1)}<0$ the relation
$x_{ik}^{c_{ik}^{(1)+}u_{pi}^++c_{ik}^{(1)-}u_{pi}^-}\Big|x_{ik}^{c_{ik}^{(1)+}u_{i}^++c_{ik}^{(1)-}u_{i}^-}$
implies $x_{ik}^{0u_{pi}^++c_{ik}^{(1)-}u_{pi}^-}\Big|x_{ik}^{0u_{i}^++c_{ik}^{(1)-}u_{i}^-}$, and then $u_{pi}^-\leq u_i^-$. Summing up, we have obtained that  $u_i=-u_i^-\leq-u_{pi}^-=u_{pi}\leq 0$. Therefore
$y_{it}^{c_{it}^{(2)+}u_{pi}^++c_{it}^{(2)-}u_{pi}^-}\Big|y_{it}^{c_{it}^{(2)+}u_{i}^++c_{it}^{(2)-}u_{i}^-}$ and thus $y_{it}^{(D_2({\bf u_p})^+)_{it}}\big|y_{it}^{(D_2({\bf u})^+)_{it}}$.

From all of the above cases we conclude that $y^{D_2({\bf u_p})^+}\big|y^{D_2({\bf u})^+}$. Therefore $y^{D_2({\bf u})^+}\in (y^{D_2({\bf u_1})^+}, \ldots,y^{D_2({\bf u_r})^+})$
and thus $\mathcal G'$ is a Gr\"obner basis of $I_{A_2}$ with respect to $<_2$.
Finally, to prove that $\mathcal G'$ is reduced we argue by contradiction.
This implies that there exists an integer $p$ such that $y^{D_2({\bf u_p})^-}\in\ini_{<_2}(I_{A_2})$,
so $y^{D_2({\bf u_p})^-}$ is divisible by some $y^{D_2({\bf u_j})^+}$.
This in turn, by a similar argument like the division argument before, implies that $x^{D_1({\bf u_j})^+}|x^{D_1({\bf u_p})^-}$, a contradiction since $\mathcal G$ is reduced.
Therefore we obtain that $\mathcal G'$ is a reduced Gr\"obner basis with respect to $<_2$, as desired.
It follows that  $f: D_1(\textbf{u}) \mapsto D_2(\bf{u})$ is a one-to-one correspondence between the universal Gr{\"o}bner bases of $I_{A_1}$ and $I_{A_2}$. \qed

\end{proof}

Note that every reduced Gr{\"o}bner basis is also a Markov basis, not necessarily minimal, thus the set of indispensables elements of $I_A$ is a subset of the intersection of all the reduced Gr{\"o}bner bases. Note
that the set of indispensables from \cite[Theorem 2.4]{OH} is the intersection of the lexicographic reduced
 Gr{\"o}bner bases or from \cite[Theorem 13]{OV} is the intersection of the degree reverse lexicographic reduced
 Gr{\"o}bner bases.  Therefore the set of indispensables elements of $I_A$ is  the intersection of all reduced Gr{\"o}bner bases. This observation and Theorem  \ref{Universal} provides a second proof of the fact that there exists a one-to-one correspondence between the indispensable binomials of the two $T_{\omega}$-robust toric ideals $I_{A_1}$ and $I_{A_2}$.

\section{Strongly Robust Ideals}
\label{section:stronglyrobust}
A binomial ${\bf x}^{{\bf u}^+}-{\bf x}^{{\bf u}^-}\in I_A$
 is called primitive if there is no other binomial ${\bf x}^{{\bf v}^+}-{\bf x}^{{\bf v}^-}\in I_A$
 such that ${\bf x}^{{\bf v}^+}$
 divides ${\bf x}^{{\bf u}^+}$ and ${\bf x}^{{\bf v}^-}$
 divides ${\bf x}^{{\bf u}^-}$. The set of the
primitive binomials is finite, it is called the Graver basis of $I_A$ and is denoted by $\Gr(I_A)$, \cite[Chapter 4]{St}. By $\Gr(A)$ we denote the set $\{ {\bf u}| {\bf x}^{{\bf u}^+}-{\bf x}^{{\bf u}^-}\in \Gr(I_A) \} $, it is the set of elements in $\Ker_\mathbb{Z}(A)$ that do not have a {\it proper conformal decomposition}.
An element
${\bf u}\in  \Ker_\mathbb{Z}(A)$ has a conformal decomposition
${\bf u}= {\bf v}+_{c} {\bf w}$ if it can  be written in the form ${\bf v}+{\bf w}$,
where ${\bf v}, {\bf w} \in  \Ker_{\ZZ}(A)$ and  ${\bf u}^+={\bf v}^++ {\bf w}^+$,  ${\bf u}^-={\bf v}^-+ {\bf w}^-$. The conformal decomposition is called proper if both
 ${\bf v}$ and  ${\bf w}$ are not zero.

A {\it strongly robust} toric ideal  is a toric ideal $I_A$ for which  the Graver basis $\Gr(I_A)$
is a minimal system of generators. Furthermore, since the ideal $I_A$ is positively graded, then any minimal system of binomial generators is
a subset of the Graver basis, see for example \cite[Theorem 2.3]{CTV2}, and therefore there is only one minimal system of generators, which is the Graver basis. We conclude that, in a strongly robust toric ideal, every Graver basis element is indispensable. As was explained at the end of Section 2, the set of indispensable elements is the intersection of all reduced Gr{\"o}bner bases, which is a subset of the Universal Gr{\"o}bner basis, which at its own is a subset of the Graver basis (see \cite[Chapter 4]{St}). Thus, we conclude that for a strongly robust toric ideal $I_A$, the
following sets are identical: the set of indispensable elements, any minimal system of generators, any reduced Gr{\"o}bner basis, the Universal Gr{\"o}bner basis and the Graver basis.

The following results will enable us to define a simplicial complex, {\em the strongly robustness complex}, which determines the strongly robustness property for
toric ideals.
According to Theorem \ref{markov}, for any two $T_{\omega}$-robust toric ideals $I_{A_1}, I_{A_2}$, we have that
 $D_1({\bf u})$ is an indispensable element of $A_1$ if and only if $D_2({\bf u})$ is an indispensable element of $A_2$, for an element ${\bf u}\in \Ker_\mathbb{Z}(T)$. Actually ${\bf u}\in \Gr(T)$, since indispensable elements of $I_{A_1}$ are in the Graver basis of $I_{A_1}$ and there is a one-to-one correspondence between $\Gr(I_T)$ and $\Gr(I_{A_1})$, given by the map ${\bf u}\mapsto D_1({\bf u})$, see \cite[Theorem 1.11]{PTV}. Therefore, the set of elements ${\bf u}$ which belong to $\Gr(T)$, such that $D({\bf u})$ is indispensable in a
 $T_{\omega}$-robust toric ideal $I_{A}$,
 does not depend on the $I_A$ chosen, but only on $T$ and $\omega$. Thus, the following definition is meaningful.
\begin{Definition}
\label{T_robust}
 Let $I_T$ be a simple toric ideal and $\omega \subseteq [s]$. We denote by $$S_{\omega}(T)=\{{\bf u}\in \Gr(T)| \ D({\bf u})\in S(A)\}$$ and call $S_{\omega}(T)$ the $T_{\omega}$-indispensable set, where $I_A$ is an
 $T_{\omega}$-robust toric ideal and $S(A)$ is the set of indispensable elements of $A$.

\end{Definition}

  In the next example we compute several sets $S_{\omega}(T)$ for different $\omega$'s. It is easy to observe an important  property of the sets $S_{\omega}(T)$, displayed in Proposition \ref{subset},
that will help us introduce a simplicial complex on the set $[s]$.

\begin{Example} \label{ExInd} {\em Let  $I_T$ be the simple toric ideal with $T=(7 \ 8 \ 9 )$. Then, using $4ti2$ \cite{4ti2}, we obtain that $\Gr(T)$ has the following 11 elements: $(1,-2,1)$, $(5,-1,-3)$, $(4,1,-4)$, $(0,9,-8)$, $(1,7,-7)$,  $(2,5,-6)$, $(3,3,-5)$, $(9,0,-7)$, $(6,-3,-2)$, $(8,-7,0)$, $(7,-5,-1)$.
 For each subset $\omega\subseteq [3]$ we compute with $4ti2$ \cite{4ti2} a Markov basis for one example of a $T_{\omega}$-robust ideal. Using the simple Algorithm 1 of \cite[Section 3]{JSC} we compute the indispensable elements for each example. Then, Theorem~\ref{markov} gives us that the $T_{\omega}$-indispensable sets for different $\omega$'s are: \\$S_{\emptyset}(T)=\Gr(T)$,\\
$S_{\{1\}}(T)=\{(1,-2,1),(5,-1,-3),(4,1,-4),  (9,0,-7),(6,-3,-2),(7,-5,-1),\\(8,-7,0)\}$, \\$S_{\{2\}}(T)=\{(1,-2,1),(5,-1,-3),(4,1,-4),(0,9,-8),(1,7,-7),(2,5,-6),(3,3,-5),  \\(6,-3,-2),(7,-5,-1),(8,-7,0)\}$, \\$S_{\{3\}}(T)=\{(1,-2,1),(5,-1,-3),(4,1,-4),(0,9,-8),(1,7,-7),(2,5,-6),(3,3,-5), \\ (9,0,-7)\}$,\\
$S_{\{1,2\}}(T)=\{(1,-2,1), (5,-1,-3), (4,1,-4), (6,-3,-2), (7,-5,-1), (8,-7,0)\}$,\\
$S_{\{1,3\}}(T)=\{(1,-2,1),(5,-1,-3),(4,1,-4),(9,0,-7)\}$,\\
$S_{\{2,3\}}(T)=\{(1,-2,1), (5,-1,-3), (4,1,-4), (0,9,-8), (1,7,-7), (2,5,-6),\\ (3,3,-5)\}$
and $S_{\{1,2,3\}}(T)=\{(1,-2,1),(5,-1,-3),(4,1,-4)\}.$
}
\end{Example}

For the next proof we will exploit  the notion of semiconformal decomposition which characterizes the indispensable elements of toric ideals (see \cite[Lemma 3.10]{HS} and \cite[Proposition 1.1]{CTV}). We recall from \cite[Definition 3.9]{HS} that for vectors ${\bf u},{\bf v},{\bf w}\in\Ker_{\mathbb{Z}}(A)$ such that ${\bf u}={\bf v}+{\bf w}$, the sum is said to be a semiconformal decomposition of ${\bf u}$, written ${\bf u}={\bf v}+_{sc} {\bf w}$, if $v_i>0$ implies that $w_i\geq 0$, and $w_i<0$ implies that $v_i\leq 0$, for all $1\leq i\leq n$.  The decomposition is called {\it proper} if both ${\bf v}, {\bf w}$ are nonzero. The set of indispensable elements ${\MS}(A)$ of $A$ consists  of all nonzero vectors in $\Ker_{\ZZ}({A})$ with no proper semiconformal decomposition, see \cite[Proposition 1.1]{CTV}.

 \begin{Proposition}\label{subset}
Let $I_T\subset K[x_1,\dots , x_s]$ be a simple toric ideal and $\omega_1\subseteq \omega_2 \subseteq [s]$.   Then we have that $S_{\omega_1}(T)\supseteq S_{\omega_2}(T)$.
\end{Proposition}

\begin{proof}
 Consider the $T_{\omega_2}$-robust toric ideal $I_{A_2}$. According to Definition~\ref{T_robust}, we have that the bouquet ideal of $I_{A_2}$ is $I_T$ and $\omega_2 =\{i\in [n]| \ {B_{i}^{(2)}}\ \text{is non-mixed}\}.$ Consider now a subset $\omega_1 \subseteq \omega_2$. Similarly, from Definition~\ref{T_robust}, the bouquet ideal of a $T_{\omega_1}$-robust toric ideal $I_{A_1}$ is $I_T$ and $\omega_1 =\{i\in [n]| \ {B_{i}^{(1)}}\ \text{is non-mixed}\}.$

Let ${\bf u} \in S_{\omega_2}(T)$, our goal is to show that ${\bf u} \in S_{\omega_1}(T)$. According to the Definition~\ref{T_robust}, applied for $S_{\omega_2}(T)$, we have $D_2({\bf u})\in S(A_2)$. Since $D_2({\bf u})\in \Ker_{\mathbb{Z}}(A_2)$, from the one-to-one correspondence between the elements of $\Ker_{\mathbb{Z}}(T)$ and the elements of each of $\Ker_{\mathbb{Z}}(A_1)$ and $\Ker_{\mathbb{Z}}(A_2)$ (see \cite[Theorem 1.9]{PTV}), we have that $D_1({\bf u})\in \Ker_{\mathbb{Z}}(A_1)$ as well.

Suppose that $D_1({\bf u}) = D_1({\bf v}) +_{sc} D_1({\bf w})$
is a semiconformal decomposition of $D_1({\bf u})$, where $D_1({\bf v}), D_1({\bf w}) \in \Ker_{\mathbb{Z}}(A_1)$. For each  $i=1,\ldots, s$,  at coordinate level we have that
$$\left (c^{(1)}_{i1}, c^{(1)}_{i2}, \ldots c^{(1)}_{ik_i} \right)u_i = \left(c^{(1)}_{i1}, c^{(1)}_{i2}, \ldots c^{(1)}_{ik_i} \right)v_i +_{sc} \left(c^{(1)}_{i1}, c^{(1)}_{i2}, \ldots c^{(1)}_{ik_i} \right)w_i ,$$
where we know that all $c^{(1)}_{i1}$ are positive, while the  $c^{(1)}_{ij}$ are all positive if the bouquet $B^{(1)}_i$ is non-mixed, and some of the $c^{(1)}_{ij}$ are negative if the bouquet $B^{(1)}_i$ is mixed.

For all $i$, we have  from  $c^{(1)}_{i1}u_i= c^{(1)}_{i1}v_i+_{sc}c^{(1)}_{i1}w_i$ that the sum $u_i= v_i+_{sc}w_i$ of the $i$-th coordinates is a semiconformal  sum, since all $c^{(1)}_{i1}>0$.
 For $i\not \in \omega_1$, we have a stronger relation, namely that the sum  $u_i= v_i+_{sc}w_i$ of the $i$-th coordinates is a conformal  sum.
 Indeed, if $B^{(1)}_i$ is free we have $u_i=v_i=w_i=0$ and the sum $0=0+0$ is conformal. Otherwise, if $B^{(1)}_i$ is mixed we distinguish three cases.

\underline{Case 1:}  If $u_i > 0$, then $c^{(1)}_{i1} u_i >0$ and the sum being semiconformal implies that $c^{(1)}_{i1} w_i \geq 0 \Rightarrow w_i \geq 0$. Since $B^{(1)}_i$ is mixed at least one of the $c^{(1)}_{ij}$, for $j\geq 2$ is negative, in which case $c^{(1)}_{ij} u_i < 0$ and from the semiconformality we have
that $c^{(1)}_{ij} v_i \leq 0 \Rightarrow v_i \geq0$. Therefore we proved that
$$
u_i>0 \Rightarrow v_i \geq 0 \text{ and } w_i \geq 0 .
$$

\underline{Case 2:} If $u_i < 0$, then $c_{i1}^{(1)} u_i <0$ and the sum being semiconformal implies that $c^{(1)}_{i1} v_i \leq 0 \Rightarrow v_i \leq 0$.
 As before, since $B^{(1)}_i$ is mixed, at least one $c_{ij}^{(1)}$ is negative, and thus  $c^{(1)}_{ij} u_i > 0$. It follows now from the semiconformality that $c^{(1)}_{ij} w_i \geq 0 \Rightarrow w_i \leq 0$, therefore
$$
u_i<0 \Rightarrow v_i \leq 0 \text{ and } w_i \leq 0 .
$$

\underline{Case 3:} Finally, if $u_i = 0$, then $c^{(1)}_{i1} u_i =0$
and the sum being semiconformal implies that $c^{(1)}_{i1} v_i \leq 0  \text{ and } c^{(1)}_{i1} w_i \geq 0$.
Since $c^{(1)}_{i1}>0$, we have that $v_i \leq 0  \text{ and } w_i \geq 0$. Since $B^{(1)}_i$ is mixed, there exists a negative $c^{(1)}_{ij}$, and we know from the semiconformality that $c^{(1)}_{ij} v_i \leq 0 \text{ and } c^{(1)}_{ij} w_i \geq0$, which  implies that $v_i \geq 0 \text{ and }  w_i \leq0$. Again we have proved that
$$
u_i=0 \Rightarrow v_i = 0 \text{ and } w_i = 0 \text{.}
$$
 The three cases analyzed before imply that for $i\not \in \omega_1$ the sum $u_i= v_i+w_i$ of the $i$-th coordinates is a conformal sum.

For $i\in \omega_2$, the sum $u_i= v_i+_{sc}w_i$ is semiconformal and stays semiconformal if it is multiplied with a positive $c^{(2)}_{ij}$.
For  $i\not \in \omega_2$, we have that $i\not \in \omega_1$, since $\omega_1\subseteq \omega_2$, therefore the sum $u_i= v_i+w_i$ is conformal and remains conformal
if it is multiplied by a positive or a negative $c^{(2)}_{ij}$. Since conformality implies semiconformality then $D_2({\bf u}) = D_2({\bf v}) +_{sc} D_2({\bf w})$, and $D_2({\bf u})\in S(A_2)$. Thus either $D_2({\bf v})={\bf 0}$ or $D_2({\bf w})={\bf 0}$, since indispensable elements do not have proper semiconformal decomposition \cite[Proposition 1.1]{CTV}.
But this in turn gives that either ${\bf v}={\bf 0}$ or ${\bf w}={\bf 0}$ which implies that either $D_1({\bf v})={\bf 0} \text{ or } D_1({\bf w}) ={\bf 0}$. This means that  $D_1({\bf u})\in S(A_1)$ or equivalently ${\bf u} \in S_{\omega_1}(T)$, as desired. \qed

\end{proof}

\begin{Definition}  The set $\Delta _T =\{\omega\subseteq [s] \ | \ S_{\omega }(T)=\Gr(T)\}$ is called the {\it strongly robustness complex} of $T$.
\end{Definition}

\begin{Corollary}
 The strongly robustness complex of $T$ is a simplicial complex.
\end{Corollary}
\begin{proof}
 Suppose that $\omega \in \Delta _T$ and let $\omega '\subseteq \omega$.
 Then $\Gr(T)=S_{\omega }(T)\subseteq S_{\omega ' }(T)\subseteq \Gr(T)$, by Proposition \ref{subset}.
 Thus $S_{\omega }=\Gr(T)$ and $\omega' \in \Delta _T$. \qed

\end{proof}
In the Example \ref{ExInd} the strongly robustness complex $\Delta _T$ is the empty set. In the case that all bouquets are mixed, Corollary 4.4 of \cite{PTV} states that the ideal is strongly robust, thus  $S_{\emptyset }(T)=\Gr(T)$ and therefore $\emptyset \in \Delta _T$.
This means that $\Delta _T$  is never the void simplicial complex.

 Boocher and Robeva posed the problem of  describing combinatorially the robustness property for toric ideals, see \cite[Question 3.5]{BR}. According to Theorem~\ref{face} below, the simplicial complex $\Delta _T$ determines the strongly robustness property for toric ideals. However, it is still an open problem whether the robustness property of toric ideals is equivalent to the
 strongly robustness property, see Section~\ref{section:codim2} for more details and results.

\begin{Theorem} \label{face} Let $I_A$ be a $T_{\omega}$-robust toric ideal. The toric ideal $I_A$ is strongly robust if and only if $\omega$ is a face of the strongly
 robustness complex $\Delta _T$.
\end{Theorem}

\begin{proof} A toric ideal $I_A$ is strongly robust  if and only if every Graver basis element is indispensable. Thus for the
 $T_{\omega}$-robust ideal $I_A$ we have: $I_A$ is strongly robust if and only if $S(I_A)=\Gr(I_A)$ if and only if $S_{\omega}(T)=\Gr(T)$ if and only if $\omega \in \Delta _T$. \qed

\end{proof}

The following theorem provides a way to compute the strongly robustness complex of a simple toric ideal $I_T$.
By $\Lambda (T)$ we denote the second Lawrence lifting of $T$, which is the $(m+s)\times 2s$ matrix
${\begin{pmatrix} T & 0 \\
I_s & I_s
\end{pmatrix}}.$ By $\Lambda (T)_{\omega}$ we denote the matrix taken from $\Lambda (T)$ by removing the $(m+i)$-th row and the $(s+i)$-th column for each $i\in \omega$. For example, for $
 T={\footnotesize\begin{pmatrix}
12 & 9 & 8 & 0 \\
0 & 3 & 4 & 12
\end{pmatrix}} $ and $\omega=\{2, 4\}$, we have that $$\Lambda (T)= {\footnotesize\begin{pmatrix}
12 & 9 & 8 & 0 & 0 & 0 & 0 & 0 \\
0 & 3 & 4 & 12 & 0 & 0 & 0 & 0\\
1 & 0 & 0 & 0 & 1 & 0 & 0 & 0 \\
0 & 1 & 0 & 0 & 0 & 1 & 0 & 0 \\
0 & 0 & 1 & 0 & 0 & 0 & 1 & 0 \\
0 & 0 & 0 & 1 & 0 & 0 & 0 & 1 \\
\end{pmatrix}} \ \ \text{ and } \ \ \Lambda (T)_{\omega}= {\footnotesize\begin{pmatrix}
12 & 9 & 8 & 0 & 0 & 0  \\
0 & 3 & 4 & 12 & 0 & 0 \\
1 & 0 & 0 & 0 & 1 & 0  \\
0 & 0 & 1 & 0 & 0  & 1 \\
\end{pmatrix}}.$$
\begin{Theorem} \label{Lawrence}  The set $\omega$ is a face of the
 strongly robustness complex $\Delta _T$ if and only if $\Lambda (T)_{\omega}$ is strongly robust.
\end{Theorem}
\begin{proof} Note that if ${\bf b}_1, \ldots, {\bf b}_s$ are the row vectors of  the Gale transform of $T$ then ${\bf b}_1,\ldots, {\bf b}_s, -\textbf{b}_1,\ldots, -{\bf b}_s$ are the row vectors of the Gale transform of $\Lambda (T)$
and ${\bf b}_1, \ldots, {\bf b}_s$, and $-{\bf b}_j$ for $j\not \in \omega$ are the row vectors of  the Gale transform of $\Lambda (T)_{\omega}$. Therefore $\Lambda (T)_{\omega}$ has $s$ bouquets of the form: $\{{\bf a}_i\}$ for $i\in \omega$, and $\{{\bf a}_j, {\bf a}_{j+s}\}$ for $j\not \in \omega$, where the indices are the indices of the columns of $\Lambda (T)$. The former are non-mixed and the later are free if ${\bf b}_j={\bf 0}$ or mixed, since ${\bf b}_{j+s}=-{\bf b}_{j}$. This means that $I_{\Lambda (T)_{\omega}}$ is a
   $T_{\omega}$-robust toric ideal.  Then $\omega \in \Delta _T$ if and only if
$\Lambda (T)_{\omega}$ is strongly robust, by Theorem \ref{face}. \qed

\end{proof}

It follows also that the strongly robustness simplicial complex depends only on the simple toric ideal $I_T$.
\begin{Example} \label{Example1}
{\em Theorem \ref{Lawrence} provides a way to construct strongly robust ideals.
Consider the $2\times 4$ integer matrix
 \[
 T={\footnotesize\begin{pmatrix}
12 & 9 & 8 & 0 \\
0 & 3 & 4 & 12
\end{pmatrix}}. \] Computing the Gale transform one  can easily see that $I_T$ is a simple toric ideal, see Example \ref{Example2}.

To find the vertices of the simplicial complex one first computes, e.g. using 4ti2 \cite{4ti2}, minimal Markov and Graver  bases of $\Lambda (T)_{\{i\}}$ for all $i\in [4]$.  It turns out that  $\Lambda (T)_{\{2\}}$, $\Lambda (T)_{\{4\}}$ are strongly robust, since in both cases the Graver basis and a minimal Markov basis given by 4ti2 are identical, while $\Lambda (T)_{\{1\}}$ and $\Lambda (T)_{\{3\}}$ are not strongly robust. Thus  $
2$ and $4$ are the only vertices of $\Delta _T$.
Finally, in this example, one has to check whether  $\{2, 4\}$ belongs or not to $\Delta _T$. It turns out that $\Lambda (T)_{\{2,4\}}$
is strongly robust, thus, the simplicial complex
$\Delta _T$  is $\{\emptyset, \{2\}, \{4\}, \{2, 4\}\}$. From Theorem~\ref{face}, all  $T_{\emptyset}, T_{\{2\}}, T_{\{4\}}, T_{\{2, 4\}}$-robust toric ideals are strongly robust, while for instance none of the $T_{\{3\}}$ or $T_{\{1,3,4\}}$-robust toric ideals are  strongly robust. To construct a strongly robust ideal take any $\omega \in  \Delta _T$, say $\{2, 4\}$, and choose four primitive integer vectors in such a way that the first and the third are mixed while the second and the fourth are not mixed. For example,  say ${\bf c}_1=(1, -1), {\bf c}_2=(5,6,7,8), {\bf c}_3=(2023, -2022, 11), {\bf c}_4=(13, 14, 15)$.
Computing any integer solution of the equation  $1=\lambda_{i1}c_{i1}+\cdots+\lambda_{im_i}c_{im_i}$ for each vector ${\bf c}_i=(c_{i1}, c_{i2}, \ldots , c_{im_i})$,
for example $1=1\cdot 1+0\cdot(-1)$, $1=-1\cdot 5+1\cdot 6+0\cdot 7+ 0\cdot 8$, $1=1\cdot 2023+1\cdot(-2022)+0\cdot(11)$, $1=0\cdot 13+(-1)\cdot 14+ 1\cdot 15$,   we get the
following generalized Lawrence matrix
\[
\tiny{
D=\
\left( \begin{array}{cccccccccccc}
 12 & 0 & -9 & 9 & 0 & 0   & 8 & 8 & 0  & 0 & 0 & 0  \\
 0 & 0 & -3 & 3 & 0 & 0   & 4 & 4 & 0  & 0 & -12 & 12  \\
 1 & 1 & 0 & 0 & 0 & 0   & 0 & 0 & 0   & 0 & 0 & 0  \\
 0 & 0 & -6 & 5 & 0 & 0   & 0 & 0 & 0   & 0 & 0 & 0  \\
 0 & 0 & -7 & 0 & 5 & 0   & 0 & 0 & 0   & 0 & 0 & 0  \\
 0 & 0 & -8 & 0 & 0 & 5   & 0 & 0 & 0   & 0 & 0 & 0  \\
 0 & 0 & 0 & 0 & 0 & 0   & 2022 & 2023 & 0    & 0 & 0 & 0  \\
 0 & 0 & 0 & 0 & 0 & 0   & -11 & 0 & 2023   & 0 & 0 & 0  \\
 0 & 0 & 0 & 0 & 0 & 0   & 0 & 0 & 0  & -14 & 13 & 0  \\
 0 & 0 & 0 & 0 & 0 & 0   & 0 & 0 & 0  & -15 & 0 & 13
\end{array} \right)  \in \ZZ^{10\times 12},}
\]
for more details  see \cite[Section 2]{PTV}.
Then $I_D$ is a $T_{\{2, 4\}}$-robust ideal, therefore it is strongly robust
by Theorem \ref{face}, since $\{2, 4\}\in \Delta _T$.

In fact, Theorem \ref{face} implies that all strongly robust ideals are taken in this way since: 1) every toric ideal is a $T_{\omega}$-robust ideal for some $T$ and $\omega$, and 2) all toric ideals have a representation in the form $I_{A'}$ where $A'$ is a generalized Lawrence matrix, see  \cite[Corollary 2.3]{PTV}.
}
\end{Example}

\section{The codimension two case}
\label{section:codim2}

Let $A$ be an integer matrix $m\times n$ of rank $n-2$ such that the toric ideal $I_A$ is positively graded. Then the columns of $A$ form an acyclic vector configuration, see \cite[Section 6.2]{Z}.  Let  $\{ {\bf b}_1, {\bf b}_2, \dots, {\bf b}_n\}$ be the Gale transform of $A$, where ${\bf b}_i=(b_{i1}, b_{i2})$ for $i=1,\ldots, n$, and $B$ the $n\times 2$ matrix with rows ${\bf b}_1, {\bf b}_2, \dots, {\bf b}_n$. Let $\widetilde{B}=\{ {\widetilde {\bf b}_1}, {\widetilde {\bf b}_2}, \dots, {\widetilde {\bf b}_n}\}$ be the reduced Gale transform of $A$, the set where ${\widetilde {\bf b}_i}=gcd(b_{i1}, b_{i2})^{-1}(-b_{i2}, b_{i1})$ if ${\bf b}_i\not =(0,0)$, and ${\widetilde {\bf b}_i}=(0,0)$ if ${\bf b}_i =(0,0)$.
Note that  ${\bf b}_i =(0,0)$ means that   ${\bf a}_i$ is a free vector. Since both configurations are totally cyclic, see \cite[Section 6.4]{Z}, we can reorder all the nonzero vectors ${\bf b}_1, {\bf b}_2, \ldots, {\bf b}_{n'}$, $n'\leq n$, in such a way that in the interior of the $\cone({\widetilde {\bf b}_i}, {\widetilde {\bf b}_{i+1}})$ there is no other ${\widetilde {\bf b}_j}$ for all $i=1,\ldots,n$, where ${\widetilde {\bf b}_{n'+1}}= {\widetilde {\bf b}_{1}}$. The minimal generating set of the monoid $\cone({\widetilde {\bf b}_i}, {\widetilde {\bf b}_{i+1}})\cap \mathbb{Z}^2$ is the
Hilbert basis of the $\cone({\widetilde {\bf b}_i}, {\widetilde {\bf b}_{i+1}})$ and is denoted by $H_i$.

\begin{Theorem} \label{Graver-indispensable} \cite[Corollary 2.5]{S}  \cite[Theorem 3.7]{PS}
 Let $A\in \mathbb{Z}^{(n-2)\times n}$ of rank $n-2$ and $Ker_{\mathbb{Z}}(A)\cap \mathbb{N}^n=\{{\bf 0}\}. $
\begin{itemize}
    \item A vector ${\bf u}\in \mathbb{Z}^2$ has either ${\bf u}$ or $-{\bf u}$ in $H_1\cup H_2\cup \cdots \cup H_{n'}$
    if and only if $x^{{\bf u}^+}-x^{{\bf u}^-}$ is in the Graver basis of $I_A$.
    \item A vector ${\bf u}\in \mathbb{Z}^2$ has both ${\bf u}$ and $-{\bf u}$ in $H_1\cup H_2\cup \cdots \cup H_{n'}$
    if and only if $x^{{\bf u}^+}-x^{{\bf u}^-}$ is an indispensable element of $I_A$.
\end{itemize}
\end{Theorem}

Let $I_A$ be a $T_{\omega}$-robust ideal of codimension 2, and ${\widetilde B}=\{{\widetilde {\bf b}_1}, \ldots, {\widetilde {\bf b}_n}\}$
be the reduced Gale transform of $A$ and $\{{\widetilde {\bf t}_1}, \ldots, {\widetilde {\bf t}_s}\}$ be the reduced Gale transform of $T$. Using the  isomorphism $D$ from $ \Ker_{\ZZ}({T})$ to $ \Ker_{\ZZ}(A)$, we get that if ${\bf t}_i=(t_{i1}, t_{i2})$ then
for an ${\bf a}_j$ that is the $k$-th element of the bouquet $B_i$ we have ${\bf b}_j=(c_{ik}t_{i1},c_{ik}t_{i2}) $. If $c_{ik}$ is positive then ${\widetilde {\bf b}_j}={\widetilde {\bf t}_i}$, and if $c_{it}$ is negative then ${\widetilde {\bf b}_j}=-{\widetilde {\bf t}_i}$.
Thus, if ${\bf a}_j$ belongs to a non-mixed or free bouquet $B_i$, then ${\widetilde {\bf b}_j}={\widetilde {\bf t}_i}$. For a mixed bouquet $B_i$, we have ${\widetilde {\bf b}_j}={\widetilde {\bf t}_i}$ for at least one ${\bf a}_j\in B_i$ and we also have ${\widetilde {\bf b}_k}=-{\widetilde {\bf t}_i}$ for at least one ${\bf a}_k\in B_i$.

\begin{Theorem} \label{polygon} Let  $I_T$ be a simple codimension 2 toric ideal. Then the strongly robustness complex $\Delta _T$ is a simplicial subcomplex of the  simplex with vertices $i\in [s]$,
such that $\widetilde{\bf t}_i$ is not a vertex of the centrally symmetric polygon $$P=\conv\left(\left\{\widetilde{\bf t}_i, -\widetilde{\bf t}_i\big| \ i\in [s]\right\}\right). $$
\end{Theorem}
\begin{proof} Let $I_A$ be a $T_{\omega}$-robust ideal,  ${\widetilde B}=\{{\widetilde {\bf b}_1}, \ldots, {\widetilde {\bf b}_n}\}$ be
 the reduced Gale transform of $A$ and ${\widetilde T}=\{{\widetilde {\bf t}_1}, \ldots, {\widetilde {\bf t}_s}\}$ be the reduced Gale transform of $T$. Then, as explained before, the set ${\widetilde B}$ is equal to $\{{\widetilde {\bf t}_1}, \ldots, {\widetilde {\bf t}_s}\}\cup\{-{\widetilde {\bf t}_i}| \ i\not \in \omega\}$. Then we claim that $\conv({\widetilde B})$ is centrally symmetric if and only if $P=\conv({\widetilde B})$. Indeed, if $P=\conv({\widetilde B})$, then $\conv({\widetilde B})$ is centrally symmetric since $P$ is centrally symmetric. Conversely, if $\conv({\widetilde B})$
is centrally symmetric, then for each ${\widetilde {\bf b}_j}\in {\widetilde B}$ we have  $-{\widetilde {\bf b}_j}\in \conv({\widetilde B})$. Thus, all ${\widetilde {\bf t}_i}$
and $-{\widetilde {\bf t}_i}$ are in $\conv({\widetilde B})$, and $P=\conv(\{\widetilde{\bf t}_i, -\widetilde{\bf t}_i| \ i\in [s]\})$
therefore $$P\subseteq \conv({\widetilde B})=\conv\left(\left\{{\widetilde {\bf t}_1}, \ldots, {\widetilde {\bf t}_s}\right\}\cup\left\{-{\widetilde {\bf t}_i}\big| i\not \in \omega\right\}\right)\subseteq \conv\left(\left \{\widetilde{\bf t}_i, -\widetilde{\bf t}_i\big| i\in [s]\right \} \right)=P,$$ thus $P=\conv({\widetilde B})$.

Let $\omega \in \Delta _T$. Then $S_{\omega}(T)=\Gr(T)$, and therefore $I_A$ is strongly robust. Using \cite[Theorem 1.2]{S} we obtain that $\conv({\widetilde B} )$ is a centrally symmetric polygon, and thus $P=\conv({\widetilde B})$. We conclude that $P=\conv\left({\widetilde B}=\left\{{\widetilde {\bf t}_1}, \ldots, {\widetilde {\bf t}_s}\right\}\cup\left\{-{\widetilde {\bf t}_i}\big|i\not\in\omega\right\} \right)$, which implies that for $i\in \omega$, $\widetilde{\bf t}_i$ is not a vertex of the centrally symmetric polygon $P$. \qed

\end{proof}

\begin{Example} \label{Example2}
{\em Theorem~\ref{polygon} provides an alternative  way to the one developed in Section~\ref{section:stronglyrobust}  to construct the strongly robustness simplicial complex for codimension 2 toric ideals.
Consider again the $2\times 4$ integer matrix
 \[
 T={\footnotesize\begin{pmatrix}
12 & 9 & 8 & 0 \\
0 & 3 & 4 & 12
\end{pmatrix}},
\]
which has the following Gale transform $B$, and whose matrix associated to the reduced Gale transform, denoted by $\widetilde{B}$, respectively,
\[
 B={\footnotesize\begin{pmatrix}
2 & 3  \\
0 & -4 \\
-3 & 0 \\
1 & 1
\end{pmatrix}}
\ \ \  \text{ and } \ \ \  \widetilde{B}={\footnotesize\begin{pmatrix}
-3 & 2  \\
1 & 0 \\
0 & -1 \\
-1 & 1
\end{pmatrix}}.
\]
Then $I_T$ is a simple toric ideal, since no two column vectors are parallel, and the  convex hull of the set $\{\widetilde{\bf b}_i, -\widetilde{\bf b}_i| \ i\in [s]\}$ is a parallelogram $P$ with vertices $\widetilde{\bf b}_1=(-3,2)$, $-\widetilde{\bf b}_1=(3,-2)$,  $-\widetilde{\bf b}_3=(0,1)$, $\widetilde{\bf b}_3=(0,-1)$. According to Theorem~\ref{polygon}, the geometry of this polygon influences the simplicial complex
$\Delta _T$, which is a subcomplex of $\{\emptyset, \{2\}, \{4\}, \{2, 4\}\}$, since $\widetilde{\bf b}_2, \widetilde{\bf b}_4$ are not vertices of $P$. However, by  \cite[Proposition 2.4]{S}, $-\widetilde{\bf b}_2=(-1,0)$ is in the Hilbert basis of the
$\cone(\widetilde{\bf b}_1, \widetilde{\bf b}_3)$ and $-\widetilde{\bf b}_4=(1,-1)$ is in the Hilbert basis of the
$\cone(\widetilde{\bf b}_3, -\widetilde{\bf b}_1)$, both of which do not involve the indices $2, 4$. Therefore, all the sets $\{2\}, \{4\}, \{2, 4\}$ are faces of $\Delta _T$. Thus, $\Delta _T=\{\emptyset, \{2\}, \{4\}, \{2, 4\}\}$.
}
\end{Example}

Finally, an ideal is called robust if the universal Gr{\" o}bner basis is a minimal generating set. It follows from \cite[Theorem 5.10]{T} that the universal Gr\" obner basis is the set of indispensables, and thus every strongly robust toric ideal is robust. In \cite{BBDLMNS}, Boocher  et al. proved that every robust toric ideal associated to a graph is strongly robust and they asked if
for every toric ideal robust implies strongly robust.  In Theorem~\ref{robust}, we prove that this is true for codimension 2 toric ideals. Note that for toric ideals of codimension 1 this is obvious, since they are principal. A key result for the proof is that circuits are always included in the Universal Gr{\" o}bner basis, see \cite[Proposition 4.11]{St}. We recall that {\em circuits} are irreducible binomials with minimal support.

\begin{Lemma} \label{circuit}
Let $I_A$ be a codimension 2 toric ideal. Then $C$ is a circuit if and only if there exists an $i$ such that $C=x^{({B{\widetilde{\bf b}^t_i})}^+}-x^{({B{\widetilde{\bf b}^t_i})}^-}$, where $B$ is the Gale transform of $A$.
\end{Lemma}
\begin{proof} Let $C$ be a circuit of  $I_A$. Then $C \in I_A$ and has no monomial factor, therefore it is in the form $x^{(B{{\bf u}^t})^+}-x^{(B{{\bf u}^t})^-}$ for some nonzero ${\bf u}\in \mathbb{Z}^2$. Since $I_A$ is a codimension 2 ideal and $C$
has minimal support, then there exists an $i\not \in \supp(C)$ and ${\bf b}_i\not =(0,0)$.
(Note that if ${\bf b}_i=(0,0)$, then ${\bf a}_i$
belongs to the free bouquet and therefore $i$ does not belong to the support of any binomial in the
form $x^{{\bf v}^+}-x^{{\bf v}^-}$, where
$ {\bf v}\in \Ker_{\ZZ}(A)$.)
Then the matrix product ${\bf b}_i \ub^t=0$ implies ${\bf u}=\lambda {\widetilde {\bf b}_i}$
and ${\bf u}, {\widetilde {\bf b}_i}$ are {\em primitive vectors}, they have the greatest common divisor of their coordinates equal to 1,  therefore $\lambda =1$ or $\lambda =-1$. Thus $C=x^{({B{\widetilde{\bf b}^t_i})}^+}-x^{({B{\widetilde{\bf b}^t_i})}^-}$.

For the converse, suppose $C_i=x^{({B{\widetilde{\bf b}^t_i})}^+}-x^{({B{\widetilde{\bf b}^t_i})}^-}$. Then for every $\ab_k\in B_i$, the $i$-th bouquet of $A$, we have ${\bf b}_k\cdot {\widetilde{\bf b}_i}=l\textbf{b}_i\cdot {\widetilde {\bf b}_i}=0$, for some $l\in \mathbb{Q}$, therefore
$k\not \in \supp(C)$. Note also from the comments before Theorem~\ref{polygon} that for $\ab_i,\ab_j$ in the same bouquet we have either ${\widetilde {\bf b}_i}= {\widetilde{\bf b}_j}$ or $ {\widetilde{\bf b}_i}= -{\widetilde{\bf b}_j}$. The complement of the support of $C_i=x^{({B{\widetilde{\bf b}^t_i})}^+}-x^{({B{\widetilde{\bf b}^t_i})}^-}$ is the union of the elements of the bouquet that contains $\ab_i$ and the elements of the free bouquet. Therefore, all the $C_i=x^{({B{\widetilde{\bf b}^t_i})}^+}-x^{({B{\widetilde{\bf b}^t_i})}^-}$ are circuits. Furthermore, the cardinality of the set of circuits is the number of non free bouquets. \qed

\end{proof}

\begin{Theorem} \label{robust}
 Let $I_A$ be a codimension 2 toric ideal. Then $I_A$ is robust if and only if it is strongly robust.
\end{Theorem}

\begin{proof}
 Strongly robust toric ideals are always robust.
 Suppose that $I_A$ is robust. From Lemma \ref{circuit} $x^{({B{\widetilde{\bf b}^t_i})}^+}-x^{({B{\widetilde{\bf b}^t_i})}^-}$ is a circuit and circuits are always  in the Universal Gr{\" o}bner basis therefore $x^{({B{\widetilde{\bf b}^t_i})}^+}-x^{({B{\widetilde{\bf b}^t_i})}^-}$ belongs to a minimal set of generators of $I_A$. By
  \cite[Theorem 5.10]{T} every generator of a robust toric ideal is indispensable. Thus also $x^{({B{\widetilde{\bf b}^t_i})}^+}-x^{({B{\widetilde{\bf b}^t_i})}^-}$ is indispensable, and therefore by Theorem \ref{Graver-indispensable}  $-\widetilde{\bf b}_i$ belongs in $H_1\cup H_2\cup \cdots \cup H_{n'}$. Applying now \cite[Lemma 2.7]{S} we obtain that $I_A$ is strongly robust. \qed

\end{proof}

In \cite{BBDLMNS} Boocher et al.  remarked on the central role that circuits play for the robustness property of toric ideals of graphs. They asked the question: if the circuits of $I_A$ are indispensable then is it true  that the toric ideal is robust? And they provide a counterexample of a toric ideal of a graph of
codimension 3 in which all circuits are indispensable but the ideal  $I_A$ is not  robust. The graph $G_t$
in Fig.\ref{8a} is a small alteration of their  counterexample, where in one of the two
cut vertices, instead of one cycle of length 3 we add $t$ cycles of length 3, $t\geq 1$. Every cycle of odd length added  on the second cut vertex increases the codimension by one. The toric ideal of the graph with $t$ cycles has codimension $2+t$.
The Graver basis, of the toric ideal of $G_t$, is identical with the Universal Gr{\" o}bner basis and consists of $t(t-1)/2+8t+4$ elements in the form
$B_{1i}=e_4e_7e_8e_{12}^2c_i-e_5e_6e_{10}^2a_ib_i$, $B_{2}=e_4e_7e_9e_{12}-e_5e_6e_{10}e_{11},$
$B_{3i}=e_1e_5^2e_{10}^2a_ib_i-e_2e_3e_7^2e_{12}^2c_i$, $B_{4i}=e_8e_{11}e_{12}c_i-e_9e_{10}a_ib_i,$
$B_{5}=e_1e_4e_5e_8-e_2e_3e_6e_7$, $B_{6i}=e_1e_4^2e_8^2e_{12}^2c_i-e_2e_3e_6^2e_{10}^2a_ib_i,$
$B_{7}=e_1e_4^2e_8e_9e_{12}-e_2e_3e_6^2e_{10}e_{11},$ $B_{8i}=e_1e_4^2e_9^2a_ib_i-e_2e_3e_6^2e_{11}^2c_i,$
$B_{9}=e_1e_5^2e_8e_{10}e_{11}-e_2e_3e_7^2e_9e_{12},$ $B_{10i}=e_4e_7e_9^2a_ib_i-e_5e_6e_8e_{11}^2c_i,$
$B_{11i}=e_1e_5^2e_8^2e_{11}^2c_i-e_2e_3e_7^2e_9^2a_ib_i,$ $B_{12i}=e_1e_4e_5e_9e_{10}a_ib_i-e_2e_3e_6e_7e_{11}e_{12}c_i,$ $B_{13(i,j)}=a_ib_ic_j-a_jb_jc_i$ where $i,j\in [t],$ see \cite{RTT, TT}.
All of them are circuits except $t$ of them in the form $B_{12i}$. The circuits are minimal and indispensable since they do not have any chords.  The $t$ elements in the form $B_{12i}$ are not, since they are not minimal, every one of them has an even chord, the $e_8$.
For details see Theorems 4.13, 4.14 in \cite{RTT} and 3.4 in \cite{TT}.
\begin{figure}[h]
\begin{center}
\includegraphics[scale=1]{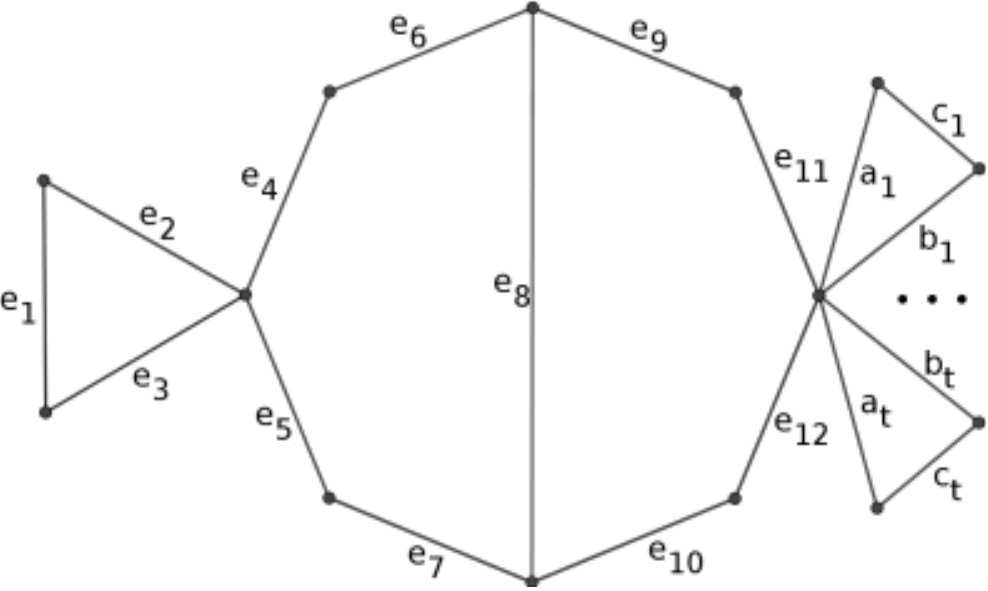}
\caption{The graph $G_t$}
\label{8a}
\end{center}
\end{figure}
All these mean that for these toric ideals $I_{G_t}$ of codimension $2+t\geq 3$ all the circuits  are indispensable and generate  the toric ideal. The $t$ elements $B_{12i}$ are not minimal but they are elements of the Universal Gr{\" o}bner basis (and the Graver basis).  We conclude that all $I_{G_t}$ are not robust.

Therefore the answer to the question is negative for  toric ideals of codimension greater than three. For codimension 1 is obviously true since in this case the ideal is principal. Next corollary shows that the answer to the question of Boocher \& al. (\cite{BBDLMNS}) about circuits and robustness  for codimension two toric ideals is positive.

\begin{Corollary}
 Let $I_A$ be a codimension 2 toric ideal. The circuits of $I_A$ are indispensable  if and only if $I_A$ is  robust.
\end{Corollary}

\begin{proof}
 It follows from Lemma \ref{circuit} that the circuits are of the form $x^{({B{\widetilde{\bf b}^t_i})}^+}-x^{({B{\widetilde{\bf b}^t_i})}^-}$. Using Theorem~\ref{Graver-indispensable} we obtain  that the circuits of $I_A$ are indispensable  if and only if  both $\widetilde{\bf b}_i$ and $-\widetilde{\bf b}_i$ are in $H_1\cup H_2\cup \cdots \cup H_n$ if and only if  $I_A$ is strongly robust, by \cite[Lemma 2.7]{S}.
  Finally, since $I_A$ is a co-dimension 2 toric ideal we have $I_A$ is strongly robust if and only if it is robust, by Theorem \ref{robust}. \qed

\end{proof}

The preceding discussions highlight the importance of better understanding the properties of the strongly robustness simplicial complex $\Delta _T$ in all codimensions. In particular, some of the key research directions to investigate are the following:
\begin{itemize}
\item Are there any classes of simple toric ideals $I_T$ for which we can determine $\Delta _T$?
    \item What kind of properties of $T$ determine $\Delta _T$?
    \item When is $\Delta _T$ different than the empty simplicial complex?
    \item What can one say about the dimension of $\Delta _T$?

\end{itemize}
Note that, in \cite{PTV}, Petrovi{\' c} et al. posed the question of whether every strongly robust ideal
has a mixed bouquet.
 In \cite{S}, Sullivant proved that codimension 2 strongly robust toric ideals have at least two mixed bouquets and posed the stronger question of whether every strongly robust ideal of codimension $r$ has $r$ mixed bouquets.
Finally, note
that,
in terms of the  strongly robustness simplicial complex $\Delta _T$, Sullivant's question becomes one about the dimension of $\Delta _T$. Namely, for every simple codimension $r$ toric ideal $I_T$ in the polynomial ring of $s$ variables, is it true that $\dim \Delta _T< s-r$?

\bigskip

{\bf Acknowledgments.} The first author has been supported by the Royal Society Dorothy Hodgkin Research Fellowship DHF$\backslash$R1$\backslash$201246. The third author has been partially supported by the grant PN-III-P4-ID-PCE-2020-0029, within PNCDI III, financed by Romanian Ministry of Research and Innovation, CNCS - UEFISCDI.

\end{document}